\documentclass[12pt]{article}
\usepackage{graphicx}
\usepackage{amsmath}
\usepackage{amssymb}
\usepackage{amsthm}
\usepackage{epsfig}
\usepackage[usenames,dvipsnames]{color}
\usepackage{setspace}
\usepackage{enumitem}
\usepackage{lipsum}
\usepackage{subcaption}
\usepackage{bm}
\usepackage{todonotes}
\usepackage[margin=2.54cm]{geometry}
\usepackage{algorithm, algorithmic} 
\usepackage{multirow}

\DeclareMathOperator{\st}{s.t.}

\DeclareMathOperator{\diag}{diag}
\DeclareMathOperator{\Diag}{Diag}

\usepackage{colortbl}
\usepackage{lipsum}
\newtheorem{theorem}{Theorem}

\newtheorem{lemma}{Lemma}

\newtheorem{assumption}{Assumption}
\newtheorem{definition}{Definition}

\newtheorem{remark}{Remark}[section]
\numberwithin{remark}{section}
\newtheorem{example}{Example}
\usepackage{dsfont}

\newcommand{\Ep}{{\mathrm{E}}}

\newcommand{\CP}{{\mathcal{CP}}}
\newcommand{\COP}{{\mathcal{COP}}}
\newcommand{\IA}{{\mathrm{IA}}}

\newcommand{\X}{{\mathcal{X}}}

\newcommand{\M}{{\mathcal{M}}}
\newcommand{\MM}{\mathrm{M}}

\newcommand{\SYM}{{\Bbb S}}
\newcommand{\PP}{\Bbb P}
\newcommand{\hP}{\widehat{\Bbb{P}}}
\newcommand{\D}{{\mathcal D}}
\newcommand{\DPe}{{\mathcal D}(\PP_N,\; \epsilon)}

\newcommand{\QQ}{\Bbb Q}
\def\RR{ {\Bbb{R}}}
\newcommand{\B}{{\mathcal B}}

\title{A Data-Driven Distributionally Robust Bound \\ 
on the Expected Optimal Value \\
of Uncertain Mixed 0-1 Linear Programming }

\author{%
Guanglin Xu\thanks{Department of Management Sciences, University of Iowa,
Iowa City, IA, 52242-1994, USA. Email: {\tt guanglin-xu@uiowa.edu}.}%
\ \ \ \ \ \ \
Samuel Burer\thanks{Department of Management Sciences, University of Iowa,
Iowa City, IA, 52242-1994, USA. Email: {\tt samuel-burer@uiowa.edu}.}%
}

\date{August 24, 2017}

\begin{document}
\maketitle

\begin{abstract}

\noindent This paper studies the expected optimal value of a mixed 0-1 programming 
problem with uncertain objective coefficients following a joint distribution. We assume that
the true distribution is not known exactly, but a set of independent samples can be
observed. Using the Wasserstein metric, we construct an ambiguity set centered at the
empirical distribution from the observed samples and containing the true distribution with 
a high statistical guarantee. The problem
of interest is to investigate the bound on the expected optimal value over the Wasserstein
ambiguity set. Under standard assumptions, we reformulate the problem into a copositive
program, which naturally leads to a tractable semidefinite-based approximation. We compare 
our approach with a moment-based approach from the literature on three applications.
Numerical results illustrate the effectiveness of our approach.

\mbox{}

\noindent Keywords: Distributionally robust optimization; Wasserstein metric; 
copositive programming; semidefinite programming
\end{abstract}

\begin{onehalfspace}

\section{Introduction} \label{sec:drlo_intro}

We consider the following uncertain mixed 0-1
linear programming problem:
\begin{equation} \label{equ:mip}
v(\xi) := \max \left\{ (F\xi)^Tx : \begin{array}{l} Ax = b, \
x \ge 0 \\ x_j \in \{0,1\} \ \forall \, j \in \B  \end{array} \right\}
\end{equation}
where $A \in \RR^{m \times n}$, $F \in \RR^{n \times k}$,
and $b \in \RR^m$ are the problem data,  $x \in \RR^n_+$ is the 
vector of decision variables, $\B \subseteq \{ 1, \ldots, n \}$
is an index set of binary variables, and the objective coefficients are linear in
the random vector $\xi \in \RR^k$ via $F$.
Problem (\ref{equ:mip}) entails two extreme
classes of programs: if $\B = \emptyset$, then
(\ref{equ:mip}) represents the regular linear program
with uncertain objective coefficients;
if $\B = \{ 1, \ldots, n \}$, then (\ref{equ:mip})
represents the regular binary program with uncertain
coefficients. In general, problem (\ref{equ:mip}) is NP-hard
\cite{Wolsey.1998}. 

The optimal value $v(\xi)$ is a random variable as $\xi$
is a random vector. We assume 
that $\xi$ follows a multivariate distribution
$\PP$ supported on a nonempty set $\Xi \subseteq \RR^k$,
which is, in particular, defined as a slice of a closed, convex, 
full-dimensional cone $\widehat \Xi \subseteq \RR_+ \times \RR^{k-1}$:
\[
\Xi := \left\{ \xi \in \widehat \Xi : e_1^T\xi =  \xi_1 = 1 \right \},
\]
where $e_1$ is the first standard basis vector in $\RR^k$. In
words, $\widehat \Xi$ is the homogenization of $\Xi$. We
choose this homogenized version for notational convenience.
Note that it, in fact, enables us to model affine effects of the uncertain
parameters in (\ref{equ:mip}). 

The expected optimal value of (\ref{equ:mip}), 
denoted by $v_{\PP}$, is defined as
\[
v_{\PP} := \Ep_{\PP} [v(\xi)] = \int_{\Xi} v(\xi) \, d\,\PP( \xi).
\]
The problem of computing $v_{\PP}$ has been extensively studied in
the literature. Hagstrom \cite{Hagstrom.1988}
showed that computing $v_{\PP}$ for the longest path problem over a
directed acyclic graph is $\#\mathcal{P}$-complete even if the arc lengths
are each independently distributed and restricted to taking two possible
values. Aldous \cite{Aldous.2001} studied a linear assignment
problem with random cost coefficients following either an independent
uniform distribution on $[0,1]$ or an exponential distribution with
parameter 1 and proved that the asymptotic value of $v_{\PP}$
approaches $\frac{\pi^2}{6}$ as the number of assignments goes
to infinity. For additional studies, see \cite{Bowman.1995, Dodin.1985, Mohring.2001}.

In practice, it is difficult or impossible to know $\PP$
completely, and computing $v_{\PP}$ is thus not well defined 
in this situation. An alternative is to construct an ambiguity set, denoted
by $\D$, that contains a family of distributions supported on $\Xi$
and consistent with any known properties of $\PP$.
Ideally, the ambiguity set will possess some statistical guarantee, 
e.g., the probability that $\PP \in \D$ will be at least $1 - \beta$, where
$\beta$ is the significance level. In analogy with $v_{\PP}$, 
we define $v_{\QQ}$
for any $\QQ \in \D$. Then, we are interested in computing 
the maximum expected optimal value $v_{\QQ}$
over the ambiguity set $\D$:
\begin{equation} \label{equ:vdp}
v_{\D}^+ :=  \sup_{\QQ \in \D} v_{\QQ}.
\end{equation}
Note that, when the probability of $\PP \in \D$ is at least
$1 - \beta$, the probability of $v_{\PP} \le v_{\D}^+$ is
at least $1 - \beta$. 

There are three main issues (somehow conflicting) regarding 
the computation of $v_{\D}^+$. First, one would like an ambiguity 
set $\D$ with a high statistical guarantee to contain the true distribution $\PP$. 
In this way, the computed $v_{\D}^+$ will be an upper bound 
on $v_{\PP}$ with a high confidence level. (We will introduce several 
approaches in the following paragraph.) Second, one would like
$v_{\D}^+$ to be tight in the sense that it is as close to
$v_{\PP}$ as possible. Generally, if $\D$ enforces more information 
about $\PP$, then $v_{\D}^+$ will be closer to $v_{\PP}$. Finally, the third 
concern is the complexity of the resulting optimization problem, 
i.e., whether the problem can be solved in polynomial time.

Bertsimas et al.~\cite{Bertsimas.Natarajan.Teo.2004,
Bertsimas.Natarajan.Teo.2006} constructed moment ambiguity sets
using the first two marginal moments of each $\xi_i$. Denote
the first and second of each uncertain parameter by $\mu_i$ and $\sigma_i$
respectively. They computed $v_{\D}^+$ over all 
joint distributions sharing the same first two marginal moments and proved 
polynomial-time computability if the corresponding deterministic 
problem is solvable in polynomial time. However, the computed bound
may not be tight with respect to $v_{\PP}$ since the marginal-moment 
model does not capture the dependence of the random variables. 
In a closely related direction, Natarajan et al.~\cite{Natarajan.Song.Teo.2009}
proposed an ambiguity set that was constructed from the
known marginal distributions of each random variable $\xi_i$,
and they computed $v_{\D}^+$ by solving a concave
maximization problem. As an extension to the marginal moment-based
approach, Natarajan et al.~\cite{Natarajan.Teo.Zheng.2011} proposed
a cross-moment model that was based on an ambiguity
set constructed using both marginal and cross moments.
Compared to the marginal-moment approach, the cross-moment approach 
has tighter upper bounds as the model captures the dependence of the random variables.
However, computing the bound requires solving a completely
positive program, which itself can only be approximated in general. 
Thus, the authors proposed semidefinite programming (SDP) 
relaxations to approximate $v_{\D}^+$.  

Moment-based ambiguity sets are also used prominently
in a parallel vein of research, called distributionally robust
optimization (DRO); see \cite{Bertsimas.Popescu.2005, Calafiore.Ghaoui.2006,
Chen.Sim.Xu.2016, Delage.Ye.2010, Dupavcova.1987,
Erdougan.Iyengar.2006, Ghaoui.Oks.Oustry.2003, 
Hanasuanto.Kuhn.Wallace.Zymler.2015, Natarajan.Teo.2017, 
Vandenberghe.Boyd.Comanor.2007, Zymler.Kuhn.Rustem.2013_1,
Zymler.Kuhn.Rustem.2013_2}. 
The popularity of the moment-based approach is mainly due to the fact that
it often leads to tractable optimization problems and relatively
simple models. Its weakness, however, is that moment-based sets
are not guaranteed to converge to the true distribution $\PP$ when 
the sample size increases to infinity, even though the
estimations of the first and second moments are themselves guaranteed
to converge. 

As an attractive alternative to moment-based ambiguity sets, 
distance-based ambiguity sets haven been proposed in recent years. 
This approach defines $\D$ as a ball in the space of probability distributions equipped with a distance
measure, and the center of the ball is typically the empirical distribution
derived from a series of independent realizations of the random vector
$\xi$. The key ingredient of this approach is the distance function.
Classical distance functions include the Kullback-Leibler
divergence~\cite{Hu.Hong.2013, Jiang.Guan.2015},
the $\phi$-divergence~\cite{Ben.Hertog.Waegenaere.Melenberg.Rennen.2013,
Duchi.Glynn.Namkoong.2016, Klabjan.Simchi.Song.2013},
the Prohorov metric~\cite{Erdougan.Iyengar.2006}, empirical Burg-entropy
divergence balls~\cite{Lam.2016}, and the 
Wasserstein metric~\cite{Pflug.Wozabal.2007, Wiesemann.Kuhn.Sim.2014}.

In this paper, we apply the Wasserstein metric to construct a data-driven
ambiguity set $\D$ centered at the empirical distribution $\widehat \PP_N$
derived from $N$ independent observations of $\xi$. This approach has
several benefits. The conservativeness of the ambiguity set
can be controlled by tuning a single parameter, the radius of the
Wasserstein ball; we will discuss this parameter in detail
in Section \ref{sec:wass}. Also, under mild conditions on $\PP$,
the Wasserstein ambiguity provides a natural confidence set for $\PP$. 
Specifically, the Wasserstein ball around the empirical distribution on 
$N$ independent identical samples contains $\PP$ with confidence $1-\beta$ 
if its radius exceeds an explicit threshold $\epsilon_N(\beta)$ that can be computed 
via a closed form equation \cite{Eichfelder.Jahn.2008, Esfahani.Khun.2015}. 
We then formulate $v_{\D}^+$ in (\ref{equ:vdp}) over the constructed 
Wasserstein ambiguity set. That is, we model the maximum value of 
$v_{\QQ}$ over the ambiguity set $\D$ constructed by the Wasserstein 
metric. In Section \ref{sec:copositive}, we reformulate problem (\ref{equ:vdp}) 
into a copositive problem under some standard assumptions. As the copositive 
reformulation is computationally intractable, we apply a standard approach 
based on semidefinite programming techniques to approximate $v_{\D}^+$ 
from above. In Section \ref{sec:example}, we numerically verify our approach 
on three applications from the literature. In particular, we compare our 
approach with the moment-based approach proposed in 
\cite{Natarajan.Teo.Zheng.2011}. We have several important 
observations from the experimental results. First, we find that the gaps 
between the bound from our semidefinite programs 
and the true expected optimal value becomes narrower as the sample size 
increases. However, the moment-based bound remains the same regardless
of the increase in the sample size. Second, we observe that our bound 
converges to the true expected optimal value on the first two applications
where the underlying deterministic problems are linear programs. Although
our bound on the third application is not able to converge to the true
expected optimal value, it is tighter than the moment-based bound after
the sample size increases to a certain level.  We conclude our research 
and discuss some future directions in Section \ref{sec:conclusion}. 

We point out some similarities of our paper to a recent technical
report by Hanasusanto and Kuhn \cite{Hanasusanto.Kuhn.2016}. In their
report, they proposed a Wasserstein-metric ambiguity set for a two-stage
DRO problem. In particular, they applied copositive programming techniques
to reformulate the second-stage worst-case value function, which is essentially
a max-min optimization problem, while we use copositive techniques to reformulate
a max-max optimization problem; see (\ref{equ:bound}). Furthermore, 
they directly used a hierarchy schema to approximate the copositive cones,
while we derive natural SDP approximations based on the copositive 
reformulation. Note that their hierarchy of approximations lead to SDP 
approximations as well. Finally, they developed an approach to derive an 
empirical Wasserstein radius, which is in spirit similar to our approach 
in this paper.

\subsection{Notation, terminology, and basic techniques}
 
We denote by $\RR^n$ the $n$-dimensional Euclidean space
and by $\RR_+^n$ the nonnegative orthant in $\RR^n$.
For a scalar $p \ge 1$, the $p$-norm of $z \in \RR^n$ is
defined $\|z\|_p := (\sum_{i=1}^n |z_i|^p)^{1/p}$, e.g., $\|z\|_1 =
\sum_{i=1}^n |z_i|$. We will drop the subscript for the $2$-norm, i.e.,
$\|z\| := \|z\|_2$. For $v,w \in \RR^n$, the inner product of $v$ and
$w$ is denoted by $v^Tw := \sum_{i=1}^n v_iw_i$. For the 
specific dimensions $k$ and $n$ of the problem in this paper, 
we denote by $e_i$ the $i$-th standard basis vector in $\RR^k$, 
and similarly, denote by $f_j$ the $j$-th standard basis vector 
in $\RR^n$. We will also define $g_1 := {e_1 \choose 0} \in \RR^{k+n}$. 
We denote by $\delta_{\xi}$ the Dirac distribution concentrating unit mass
at $\xi \in \RR^k$. For any $N \in \mathbb{N}$, 
we define $[N] := \{1, \ldots, N\}$. 

Let $\RR^{m \times n}$ denote the space of real $m \times n$
matrices, and $A\bullet B := \text{trace}(A^TB)$ denote the 
trace of the inner product of two matrices $A, B \in \RR^{m \times n}$. 
We denote by $\SYM^n$ the space of $n \times n$ symmetric matrices, 
and for $X \in \SYM^n$, $X \succeq 0$ represents that $X$ is positive 
semidefinite. In addition, we denote by $\diag(X)$ the vector containing 
the diagonal entries of $X$, and denote by $\Diag(v)$ the diagonal 
matrix with vector $v$ along its diagonal. $I \in \SYM^n$ denotes 
the identity matrix. 

Finally, letting ${\cal K} \subseteq \RR^n$ be a closed,
convex cone, and ${\cal K}^*$ be its dual cone, we give
a brief introduction to {\em copositive programming} with
respect to the cone ${\cal K}$. The {\em copositive cone\/} with respect
to $\cal K$ is defined as
\[
    \COP({\cal K}) := \{ M \in \SYM^n: x^T M x \ge 0 \ \forall \ x \in {\cal K} \},
\]
and its dual cone, the {\em completely positive cone\/} with respect to
${\cal K}$, is given as
\[
    \CP({\cal K}) := \{ X \in \SYM^n: X = \textstyle{\sum_i} x^i (x^i)^T, \ x^i \in {\cal K} \},
\]
where the summation over $i$ is finite but its cardinality is
unspecified. The term {\em copositive programming\/} refers to linear
optimization over $\COP({\cal K})$ or, via duality, linear optimization
over $\CP({\cal K})$. In fact, these problems are sometimes called
{\em generalized copositive programming\/} or {\em set-semidefinite
optimization\/} \cite{Burer.Dong.2012,Eichfelder.Jahn.2008} in contrast
with the standard case ${\cal K} = \RR_+^n$. In this paper, we work with
generalized copositive programming, although we use the shorter phrase
for convenience.

\section{A Wasserstein-Based Ambiguity Set} \label{sec:wass}

In this section, we define the Wasserstein metric and discuss 
a standard method to construct a Wasserstein-based ambiguity set.
Using this ambiguity set, we fully specify problem (\ref{equ:vdp}). 

Denote by ${\widehat \Theta}_N := \{\hat \xi^1, \ldots, \hat \xi^N\}$ 
the set of $N$ independent samples of $\xi$ 
governed by $\PP$. The uniform empirical 
distribution based on ${\widehat \Theta}_N$ is 
$\widehat \PP_N := \dfrac1N \sum_{i=1}^N \delta_{\hat \xi^i}$
where $\delta_{\zeta}$ is the Dirac distribution concentrating
unit mass at $\zeta \in \RR^k$.

\begin{definition} [Definition 3 in \cite{Hanasusanto.Kuhn.2016}]
Let $\M^2(\Xi)$ be the set of all probability distributions
$\QQ$ that are supported on $\Xi$ and that satisfy $\Ep_{\QQ}[\|\xi 
- \xi'\|^2] = \int_{\Xi}\|\xi - \xi'\|^2 d\,\QQ(\xi) < \infty$ where 
$\xi' \in \Xi$ is some reference point, e.g., $\xi' = \hat \xi^i$ for some $i \in [N]$. 
\end{definition}

\begin{definition} [Definition 3 in \cite{Hanasusanto.Kuhn.2016}]
The $2$-Wasserstein distance between any $\QQ, \QQ' \in {\cal M}^2(\Xi)$
is  
\[
W^2(\QQ, \QQ') := \inf \left\{ \left (\int_{\Xi^2} \|\xi - \xi'\|^2 \, \Pi (d\xi, d\xi') \right )^{1/2} : 
\begin{array}{l} \text{$\Pi$ is a joint distribution of $\xi$ and $\xi'$} \\
                       \text{with marginals $\QQ$ and $\QQ'$, respectively}  \end{array} \right\}.
\]
\end{definition} 

\begin{remark}
The Wasserstein distance is essentially the minimum cost of redistributing mass
from $\QQ$ to $\QQ'$. It is also called the ``earth mover's distance" in
the community of computer science; see \cite{Rubner.Tomasi.Guibas.2000}.
In fact, the Wasserstein distance between two discrete distributions with 
a finite number of positive masses corresponds to a transportation planning 
problem in finite dimensions. 
\end{remark}

Example \ref{ex:ex1} illustrates the Wasserstein distance between two
discrete distributions.
\begin{example} \label{ex:ex1}
Consider two discrete distributions: $\QQ := \sum_{i=1}^Mq_i \delta_{\xi_i}$
and $\QQ' := \sum_{j=1}^{M'} q'_j\delta_{\xi'_j}$ where $q_i \geq 0  \ i = 1, \ldots, M$, 
$q'_j \geq 0 \ j = 1, \ldots, M'$, and $\sum_{i=1}^M q_i = \sum_{j=1}^{M'}q'_j = 1$.
Define $c_{ij} = \|\xi_i - \xi'_j\|^2 \ \forall \; i = 1, \ldots, M \  j = 1, \ldots, M'$.  
Then, the $2$-Wasserstein distance between $\QQ$ and $\QQ'$ equals the square root
of the optimal value of the following transportation planning problem:
\begin{equation}
\begin{array}{ll}
\min \limits_{\pi}& \sum_{i=1}^M \sum_{j=1}^{M'} c_{ij} \pi_{ij} \\
\st    & \sum_{j=1}^{M'} \pi_{ij} = q_i \ \forall \; i = 1, \ldots, M \\
        & \sum_{i=1}^{M} \pi_{ij} = q'_j \ \forall \; j = 1, \ldots, M' \\
        & \pi_{ij} \geq 0 \ \forall \; i = 1, \ldots, M, \ j = 1, \ldots, M',
\end{array}
\end{equation}
where $\pi$ is the joint distribution of $\xi$ and $\xi'$ with marginals of $\QQ$ and $\QQ'$
and $\pi$ is the matrix variable in this optimization problem. 

\end{example}

With this setting, our ambiguity set contains a family of 
distributions that are close to $\widehat \PP_N$ with respect to the 
Wasserstein metric. In particular, we define our ambiguity set $\D$ 
as a $2$-Wasserstein ball of radius $\epsilon$ that is centered at the 
uniform empirical distribution $\widehat \PP_N$:
\[
\D(\widehat \PP_N, \epsilon) := \left\{ \QQ \in \M^2(\Xi) :  W^2(\QQ, \widehat \PP_N) \le \epsilon \right\}.
\]
The reader is referred to \cite{Hanasusanto.Kuhn.2016} for 
the general case of $\M^r(\Xi)$ and $W^r(\QQ, \QQ')$ for any
$r \geq 1$. We use the $2$-Wasserstein distance in this paper for  
two reasons. First, the Euclidean distance is one of 
the most popular distances considered in the relevant literature;
see \cite{Esfahani.Khun.2015, Hanasusanto.Kuhn.2016}. 
Second, we will find that problem (\ref{equ:vdp}) with an ambiguity 
set based on the $2$-Wasserstein distance can be reformulated into 
a copositive program; see Section \ref{sec:copositive}. 

Then, we replace the generic ambiguity set $\D$ with the Wasserstein 
ball $\D(\hP_N, \epsilon)$ in problem (\ref{equ:vdp}) to compute
a data-driven upper bound: 
\begin{equation} \label{equ:bound}
\begin{array}{lll} v_{\D(\hP_N, \epsilon)}^+ = &  \sup \limits_{\Pi, \ \QQ \in \M^2(\Xi)}  &  \int_{\Xi}  v(\xi) \ d\,\QQ(\xi)  \\ 
                                           &  \st &  \int_{\Xi^2} \|\xi - \xi'\|^2 \ \Pi(d\xi, d\xi') \le \epsilon^2 \\
                                           & & \text{$\Pi$ is a joint distribution of $\xi$ and $\xi'$} \\
                                           & &  \text{with marginals $\QQ$ and $\widehat \PP_N$, respectively}.              
\end{array}
\end{equation}

We next close this subsection by making some remarks.
First, the Wasserstein ball radius in problem (\ref{equ:bound})
controls the conservatism of the optimal value. A larger radius 
is more likely to contain the true distribution and thus a more likely 
valid upper bound on $v_{\PP}$, but even if it is valid, it could be a 
weaker upper bound. Therefore, it is crucial to choose an 
appropriate radius for the Wasserstein ball. Second, the 
Kullback-Leibler divergence ball is also considered in recent research; 
see \cite{Hu.Hong.2013, Jiang.Guan.2015}. However, in the case of
our discrete empirical distribution, the Kullback-Leibler divergence ball is a 
singleton containing only the empirical distribution itself, with probability 
one. Third, the ambiguity sets constructed by goodness-of-fit 
tests in \cite{Bertsimas.Gupta.Kallus.2013, Bertsimas.Gupta.Kallus.2014}
also possess statistical guarantees, however, they often lead to
complicated and intractable optimization problems for the case of 
high-dimensional uncertain parameters.


\subsection{An empirical Wasserstein radius}
The papers \cite{Eichfelder.Jahn.2008, Esfahani.Khun.2015} 
present a theoretical radius $\epsilon_N(\beta)$ for 
datasets of size $N$, which guarantees a desired confidence 
level $1 - \beta$ for $\PP \in \D(\PP_N, \epsilon_N(\beta))$ 
under the following standard assumption on $\PP$:
\begin{assumption}[Light-tailed distribution] There exists an
exponent $a > 1$ such that 
\[
\Ep_{\PP} [\exp(\|\xi\|^a)] = \int_{\Xi} \exp(\|\xi\|^a)d\,\PP(\xi) <  \infty.
\]
\end{assumption}
\noindent Note that $\epsilon_N(\beta)$ depends on $N$ 
and $\beta$. However, $\epsilon_N(\beta)$ is known to be 
conservative in practice; see \cite{Eichfelder.Jahn.2008} for 
example. In other words, $\D(\PP_N, \epsilon_N(\beta))$ 
might contain significantly more irrelevant distributions so 
that the computed $v_{\D(\PP_N, \; \epsilon_N(\beta))}^+$ 
is significantly larger than $v_{\PP}$. So, we propose a 
procedure to derive an empirical radius that provides 
a desired confidence level $1 - \beta$ but is much smaller 
than $\epsilon_N(\beta)$. Our approach is based on 
the data set $\widehat \Theta_N$. In particular, we 
apply a procedure, similar to cross validation in spirit, 
that computes an empirical confidence level (between 
$0$ and $1$) for a given radius $\epsilon$; see details 
in the next paragraphs. Our procedure guarantees that 
a larger radius leads to a higher confidence level. 
Therefore, by iteratively testing different $\epsilon$, 
we can find a radius with a desired confidence level based 
on the data set $\widehat \Theta_N$. Although the derived 
$\epsilon({\widehat \Theta}_N, \, \beta)$ depends on the 
data set $\widehat \Theta_N$, our experimental results in 
Section \ref{sec:example} indicate that it can be used 
for other datasets of  the same sample size. We will 
show the numerical evidence in Section \ref{sec:example}. 
Our approach is also similar in spirit to the one used 
in \cite{Esfahani.Khun.2015, Hanasusanto.Kuhn.2016}.

Our procedure requires an oracle to compute (or approximate) 
$v_{\DPe}^+$. Later in Section \ref{sec:copositive}, we will propose 
a specific approximation; see (\ref{equ:vdp_sdp}). Assume also that, 
in addition to the dataset ${\widehat \Theta}_N$, we predetermine a 
set $\mathcal E$ containing a large, 
yet finite, number of candidate radii $\epsilon$. We randomly divide 
${\widehat \Theta}_N$ into training and validation datasets $K$ times.  
We enforce the same dataset size denoted by $N_T$ on each of the 
$K$ training datasets. 

Next, for each $\epsilon \in \mathcal E$, we derive an empirical 
probability based on the following procedure: (i) we use each of 
the $K$ training datasets to approximate 
$v_{\D({\widehat \PP}_{N_T}, \; \epsilon)}^+$ with a value called 
$v_{\text{WB}}(\epsilon)$ by calling the oracle, where $\PP_{N_T}$
represents the empirical distribution from the training set; (ii) we then use 
the corresponding $K$ validation datasets to simulate the expected 
optimal values denoted by $v_{\text{SB}}$ \footnote{This process 
is to solve a linear program or integer program corresponding to 
each sample in the validation dataset and then to take the average 
of the optimal values.}; and (iii) we finally compute 
the percentage of the $K$ instances where $v_{\text{WB}}(\epsilon) 
\ge v_{\text{SB}}$. Let us call this empirical probability as the 
{\em empirical confidence level}. Thus, the empirical confidence level 
can roughly approximate the confidence level that the underlying 
distribution is contained in the Wasserstein-based ambiguity 
set with the radius $\epsilon$. Note that the percentage 
computed is non-decreasing in $\epsilon$ and equal to $1$ for 
some large $\epsilon_0$. Therefore, the set containing all the 
empirical confidence levels is essentially an empirical cumulative distribution. 
Then, given a desired confidence level, we can 
choose a corresponding empirical radius $\epsilon \in \mathcal E$. 
The numerical results in Section \ref{sec:example} indicate that
our choices of $\epsilon$ indeed return the desired confidence levels.  
We specify the above procedure in Algorithm \ref{algorithm:a1}.
\begin{algorithm}
\caption{Procedure to compute an empirical confidence level for any $\epsilon \in {\cal E}$} \label{algorithm:a1}
\begin{algorithmic}
\STATE \textbf{Inputs:} A dataset ${\widehat \Theta}_N= \{\hat \xi^1, \ldots, \hat \xi^N\}$ and a radius $\epsilon \in \mathcal E$
\STATE \textbf{Outputs:} The empirical confidence level 
\FOR{$k=1,\ldots,K$}
\STATE Use the $k^{\text{th}}$ training dataset to compute $v^k_{\text{WB}}(\epsilon)$ 
\STATE Use the $k^{\text{th}}$ validation dataset to simulate $v^k_{\text{SB}}$
\ENDFOR
\STATE Calculate the empirical confidence level for $\epsilon$ as the percentage of 
the $K$ instances where $v^k_{\text{WB}}(\epsilon) \ge v^k_{\text{SB}}$
\end{algorithmic}
\end{algorithm}

\section{Problem Reformulation and Tractable Bound} \label{sec:copositive}

In this section, we propose a copositive programming reformulation
for problem (\ref{equ:bound}) under some mild assumptions. 
As copositive programs are computationally intractable, we then propose
semidefinite-based relaxations for the purposes of computation. 

Let us first define the feasible set for $x \in \RR^n$  in (\ref{equ:mip}) 
as follows:  
\[
\X := \left\{ x \in \RR^n :  \begin{array}{l} Ax = b, \
x \ge 0 \\ x_j \in \{0,1\} \ \forall \, j \in \B  \end{array}  \right\}.
\]
We now introduce the following standard assumptions: 

\begin{assumption} \label{ass:feasible_bound}
The set $\X \subseteq \RR^n$ is nonempty and bounded. 
\end{assumption}

\begin{assumption} \label{ass:binary}
$Ax = b, \ x \ge 0 \implies 0 \le x_j \le 1 \ \forall \, j \in \B$. 
\end{assumption}
Assumption \ref{ass:binary} can be easily enforced. 
For example (see also \cite{Burer.2009}, \cite{Natarajan.Teo.Zheng.2011}), 
if $\B = \emptyset$, then the assumption is redundant; if problem (\ref{equ:mip})
is derived from the network flow problems, for instance
the longest path problem on a directed acyclic graph, then
Assumption \ref{ass:binary} is implied from the network
flow constraints; if $\B$ is a nonempty set and the 
assumption is not implied by the constraints, we can add
constraints $x_j + s_j = 1, \ s_j \ge 0 \ \forall \, j \in \B$.

\begin{assumption} \label{ass:convex}
The support set $\Xi \subseteq \RR^k$ is convex, closed, 
and computationally tractable.
\end{assumption}

\noindent For example, $\Xi$􏰐 could be represented using a polynomial 
number of linear, second-order-cone, and semidefinite inequalities. 
In particular, the set $\Xi$ possesses a polynomial-time separation 
oracle \cite{Grotschel.Lovasz.Shrijver.1981}.

\begin{assumption}  \label{ass:uncertainty_bound}
$\Xi$ is bounded. 
\end{assumption}

By Assumption \ref{ass:feasible_bound}, we know that
$v(\xi)$ is finite and attainable for any $\xi \in \Xi$.
Note that under Assumptions 
\ref{ass:feasible_bound}-\ref{ass:uncertainty_bound}, 
$v_{\D}^+$ is finite and attainable and thus we can 
replace $\sup$ with $\max$ in (\ref{equ:bound}) under 
these conditions. Assumption \ref{ass:uncertainty_bound} 
could be merged with Assumption \ref{ass:convex}, but 
it is stated separately to highlight its role in proving the 
exactness of the copositive programming reformulation 
below.

\subsection{A copositive reformulation} \label{sec:ubd}

We reformulate problem (\ref{equ:bound}) via conic programming duality
theory and probability theory. We introduce a useful result from the literature
as follows.
\begin{lemma} \label{lem:moment}
$v_{\D(\hP_N, \epsilon)}^+$ equals the optimal value of
\begin{equation} \label{equ:moment}
\begin{array}{ll}   \sup  & \frac1N  \sum_{i=1}^N \int_{\Xi}  v(\xi) \ d\,\QQ_i(\xi)  \\ 
                            \st &  \frac1N \sum_{i=1}^N \int_{\Xi} \|\xi - \hat \xi_i\|^2 \ d\,\QQ_i(\xi) \le \epsilon^2 \\
                                 & \QQ_i \in {\cal M}^2(\Xi) \ \ \ \forall \, i \in [N],    
\end{array}
\end{equation}
where $\QQ_i$ represents the distribution of $\xi$ conditional on $\xi' = \hat \xi^i$ for all $i \in [N]$.  
\end{lemma}
\begin{proof}
As $\QQ_i$ represents the distribution of $\xi$ conditional on 
$\xi' = \hat \xi^i$, the joint probability $\Pi$ in problem (\ref{equ:bound}) 
can be decomposed as $\Pi = \frac1N \sum \limits_{i \in [N]} \QQ_i$ 
by the law of total probability. Thus, the optimal value of (\ref{equ:moment})
coincides with $v_{\D(\hP_N, \epsilon)}^+$, which completes the proof.
 \end{proof}

We next provide a copositive programming reformulation
for problem (\ref{equ:moment}). As the first step, we use
a standard duality argument to write the dual of (\ref{equ:moment})
(see also \cite{Esfahani.Khun.2015}):
\begin{align}
  v_{\D(\hP_N, \epsilon)}^+ = & \sup \limits_{\QQ_i \in \MM^2(\Xi)} \inf \limits_{\lambda \ge 0} \frac{1}{N}\sum_{i=1}^N \int_{\Xi} v(\xi) \ d\,\QQ_i(\xi) + \lambda \left (\epsilon^2 - \frac{1}{N}\sum_{i=1}^N \int_{\Xi}  \|\xi - \hat \xi_i\|^2 \ d\,\QQ_i(\xi) \right ) \\
   \le & \inf \limits_{\lambda \ge 0} \sup \limits_{\QQ_i \in \MM^2(\Xi)} \lambda \, \epsilon^2 + \frac{1}{N}\sum_{i=1}^N \int_{\Xi} (v(\xi) - \lambda \|\xi - \hat \xi_i\|^2) \ d\,\QQ_i(\xi) \label{equ:sub1+} \\
   = & \inf \limits_{\lambda \ge 0} \lambda \, \epsilon^2 + \frac{1}{N}\sum_{i=1}^N \sup \limits_{\xi \in \Xi} \ (v(\xi) - \lambda \|\xi - \hat \xi_i\|^2), \label{equ:sub2+}
\end{align}
where (\ref{equ:sub1+}) follows from the max-min inequality, 
while equation (\ref{equ:sub2+}) follows from the
fact that ${\cal M}^2(\Xi)$ contains all the Dirac distributions supported on $\Xi$. 



By Assumption \ref{ass:feasible_bound}, $v(\xi)$ is finite for all $\xi \in \Xi$.
Then, the inequality in (\ref{equ:sub1+}) becomes an equality for any
$\epsilon > 0$ due to a straightforward generalization of a strong duality 
result for moment problems in Proposition 3.4 in \cite{Shapiro.2001}; see
also Theorem 1 in \cite{Hanasusanto.Kuhn.2016} and Lemma 7 in 
\cite{Hanasusanto.Roitch.Kuhn.Wiesemann.2017}. 
By introducing auxiliary variables $s_i$, the minimization problem in (\ref{equ:sub2+})
is equivalent to 
\begin{equation} \label{equ:vdp_nonconvex}
\begin{array}{lll}
v_{\D(\hP_N, \epsilon)}^+ = & \inf \limits_{\lambda, \, s_i} & \lambda \, \epsilon^2 + \frac1N \sum_{i=1}^{N} s_i \\ 
                & \st &  \sup \limits_{\xi \in \Xi} \ (v(\xi) - \lambda \|\xi - \hat \xi_i\|^2) \le s_i \ \ \ \forall \, i \in [N] \\ 
                & & \lambda \ge 0.
\end{array}
\end{equation}
For each $i \in [N]$, consider the following maximization problem
corresponding to the left-hand side of the constraints in (\ref{equ:vdp_nonconvex}):
\begin{equation} \label{equ:inner_blp}
\begin{array}{lll} 
h^i(\lambda) := & \sup & (F\xi)^Tx - \lambda(\xi^T\xi - 2\hat \xi_i^T\xi + \|\hat \xi_i\|^2) \\
                         & \st    & Ax = b, \ x \ge 0 \\
                         &         & x_j \in \{0, 1\} \ \ \ \forall \, j \in \B \\ 
                         &         &  e_1^T\xi = 1, \ \xi \in \hat \Xi,
\end{array}
\end{equation}
which is a mixed 0-1 bilinear program. Under Assumption \ref{ass:binary}, 
it holds also that the optimal value of (\ref{equ:inner_blp}) equals the optimal value of
an associated copositive program \cite{Burer.2009, Burer.2012}, 
which we now describe. 

Define
\begin{equation} \label{equ:zE}
    z := {\xi \choose x} \in \RR^{k + n}, \quad
    E := \begin{pmatrix} -be_1^T & A \end{pmatrix} \in \RR^{m \times (k + n)},
\end{equation}
\begin{equation} \label{equ:H_lam}
    H^i(\lambda) := \begin{pmatrix} -\lambda (I - \hat \xi_i e_1^T - e_1\hat \xi_i^T + \|\hat \xi_i\|^2e_1e_1^T) & \frac12 F^T \\ \frac12 F & 0 \end{pmatrix}  \in \SYM^{k+n},
\end{equation}
and for any $j \in \B$, define
\begin{equation} \label{equ:zQ}
 Q_j :=  {0 \choose f_j}{0 \choose f_j}^T - \frac12 {0 \choose f_j}{e_1 \choose 0}^T 
              -  \frac12 {e_1 \choose 0} {0 \choose f_j}^T \in \SYM^{k+n}.
\end{equation} 
where $f_j$ denotes the $j$-th standard basis vector in $\RR^n$. 

Because both $\X$ and $\Xi$ are bounded by Assumptions \ref{ass:feasible_bound}
and \ref{ass:uncertainty_bound}, there exists a scalar $r > 0$ such that the constraint
$z^T z = \xi^T \xi + x^T x \le r$ is redundant for (\ref{equ:inner_blp}). 
Furthermore, it is well-known that we can use the following
quadratic constraints to represent the binary variables
in the description of $\X$: 
\[
x_j^2 - x_j = 0 \ \ \ \Leftrightarrow \ \ \ Q_j\bullet zz^T = 0.
\] 
After adding the redundant constraint and representing the
binary variables, we homogenize problem (\ref{equ:inner_blp})
as follows:
\begin{equation} \label{equ:inner_hom}
\begin{array}{ll} 
\max & H^i(\lambda) \bullet zz^T \\
\st    & Ez = 0, \ g_1^Tz = 1 \\
        & I \bullet zz^T \le r  \\ 
        & Q_j \bullet zz^T = 0 \ \ \ \forall \, j \in \B \\ 
        & z \in \widehat \Xi \times \RR_+^n,
\end{array}
\end{equation}
where $g_1 = {e_1 \choose 0} \in \RR^{k+n}$
and $e_1$ denotes the standard basis 
vector in $\RR^k$.  
The copositive representation is thus
\begin{equation} \label{equ:inner_cpp}
\begin{array}{ll} 
\max \;\; & H^i(\lambda) \bullet Z  \\
\st    \;\; &  \diag(EZE^T) = 0  \\
        \;\; &  g_1 g_1^T \bullet Z = 1 \\
        \;\; &  I \bullet Z \le r \\ 
        \;\; & Q_j \bullet Z = 0 \ \ \ \forall \, j \in \B \\ 
        \;\; & Z \in \CP(\widehat \Xi \times \RR_+^n).  
\end{array}
\end{equation}
Letting $u^i \in \RR^m$, $\rho^i \in \RR_+$, $\alpha^i \in \RR$, and
$v^i \in \RR^{|\B|}$ be the respective dual multipliers of
$\diag(EZE^T) = 0$, $I \bullet Z \le r$, $g_1 g_1^T \bullet Z = 1$,
and $Q_j \bullet Z = 0$, standard conic duality
theory implies the dual of (\ref{equ:inner_cpp}) is
\begin{equation} \label{equ:inner_cop}
\begin{array}{ll}
\min \limits_{\alpha^i, \rho^i, u^i, v^i} & \alpha^i + r \rho^i \\
\st &  \alpha^i g_1 g_1^T - H^i(\lambda) + E^T \Diag(u^i)E + \sum \limits_{j \in \B} v^i_jQ_j + \rho^i I \in \COP(\widehat \Xi \times \RR_+^n) \\
& \rho^i \ge 0.
\end{array}
\end{equation}
Holding all other dual variables fixed, for $\rho^i > 0$ large,
the matrix variable in (\ref{equ:inner_cop}) is strictly
copositive---in fact, positive definite---which establishes that
Slater's condition is satisfied, thus ensuring strong duality:
the optimal value of (\ref{equ:inner_cpp}) equals
the optimal value of (\ref{equ:inner_cop}). Therefore,
we can reformulate problem (\ref{equ:vdp_nonconvex})
as follows:
{\small
\begin{equation} \label{equ:vdp_cop}
\begin{array}{lll}
v_{\D(\hP_N, \epsilon)}^+ = & \min & \lambda \epsilon^2 + \frac1N \sum \limits_{i=1}^{N} (\alpha^i + r\rho^i) \\ 
                & \st &  \alpha^i g_1 g_1^T - H^i(\lambda) + E^T \Diag(u^i)E + \sum \limits_{j \in \B} v^i_{j}Q_j + \rho^i I \in \COP(\widehat \Xi \times \RR_+^n) \ \forall \, i \in [N] \\ 
                & & \rho^i \ge 0  \ \ \ \forall \, i \in [N] \\
                & & \lambda \ge 0.
\end{array}
\end{equation}
}

Note that if Assumption \ref{ass:uncertainty_bound} fails, the constraint
$I \bullet Z \le r$ should be excluded from (\ref{equ:inner_cpp})
and thus the terms $r\rho^i$ and $\rho^i I$ in the objective
function and the constraint, respectively, should be excluded 
in (\ref{equ:inner_cop}) as well. As such, strong duality 
between (\ref{equ:inner_cpp}) and (\ref{equ:inner_cop}) cannot be
established in this case. However, the modified (\ref{equ:inner_cop}) still
provides an upper bound on $h^i(\lambda)$.  Accordingly,
the modified problem (\ref{equ:vdp_cop}) still provides an
upper bound on $v_{\D(\hP_N, \epsilon)}^+$.

\subsection{A semidefinite-based relaxation} 

As problem (\ref{equ:vdp_cop}) is difficult to solve in general, we 
propose a tractable approximation based on semidefinite programming
techniques. In particular, we propose an inner approximation
of $\COP(\widehat \Xi \times \RR_+^n)$ in
(\ref{equ:vdp_cop}) so that the resulting problem has an optimal
value that is an upper bound on $v_{\D}^+$. Now, define
\[
    \IA(\widehat \Xi \times \RR_+^n) :=
    \left\{
        S + M
        \ : \
        \begin{array}{c}
         S_{11} \in \IA(\widehat \Xi), \text{Rows}(S_{21}) \in \widehat \Xi^* \\
            S_{22} \ge 0, \ M \succeq 0 
        \end{array}
    \right\},
\]
where $\IA(\widehat \Xi)$ is an inner approximation of $\COP(\widehat \Xi)$, i.e., 
$\IA(\widehat \Xi) \subseteq \COP(\widehat \Xi)$. 
Immediately, we have a relationship between $ \IA(\widehat \Xi \times \RR_+^n)$
and $ \COP(\widehat \Xi \times \RR_+^n)$:
\begin{lemma} \label{lem:inner}
$\IA(\widehat \Xi \times \RR_+^n) \subseteq \COP(\widehat \Xi \times \RR_+^n)$.
\end{lemma} 

\begin{proof}
Let arbitrary ${p \choose q} \in \widehat \Xi \times \RR_+^n$ be given. We need
to show
\[
    {p \choose q}^T
    \left(
        S + M
    \right)
    {p \choose q} = \textstyle{{p \choose q}}^T S {p \choose q} +
    \textstyle{{p \choose q}}^T M {p \choose q} \ge 0.
\]
\begin{align}
 {p \choose q}^T
    \left(
        S + M
    \right)
    {p \choose q} = & \  \textstyle{{p \choose q}}^T S {p \choose q} +
    \textstyle{{p \choose q}}^T M {p \choose q} \\
   = & \ p^T S_{11}p + 2q^TS_{21}p + q^TS_{22}q +  \textstyle{{p \choose q}}^T M {p \choose q} \\
  \ge & \ 0
\end{align}
The first term is nonnegative because $p \in \widehat \Xi$ and
$S_{11} \in \IA(\widehat \Xi) \subseteq \COP(\widehat \Xi)$; the second
term is nonnegative because $p \in \widehat \Xi, q \ge 0$, and
$\text{Rows}(S_{21}) \in \widehat \Xi^*$; the third term is
nonnegative because $q \ge 0$ and $S_{22} \ge 0$; the
last term is nonnegative because $M \succeq 0$. 
\end{proof}
When $\widehat \Xi = \{\xi \in \RR^k: P\xi \ge 0\}$ is a polyhedral cone based
on some matrix $P \in \RR^{p \times k}$, a typical inner approximation $\IA(\widehat \Xi)$
of $\COP(\widehat \Xi)$ is given by
\[
\IA(\widehat \Xi) := \{ S_{11} = P^T Y P : Y \ge 0 \},
\]
where $Y \in \SYM^p$ is a symmetric matrix variable. This corresponds
to the RLT approach of \cite{Anstreicher.2009, Burer.2015, Sherali.Adams.2013}.
When $\widehat \Xi= \{ \xi \in \RR^k : \|(\xi_2, \ldots, \xi_k)^T\| \le \xi_1 \}$ is
the second-order cone,
it is known \cite{Sturm.Zhang.2003} that
\[
   \COP(\widehat \Xi) = \{ S_{11} = \tau J + M_{11} : \tau \ge 0, \ M_{11} \succeq 0 \},
\]
where $J = \Diag(1, -1, \ldots, -1)$. Because of this simple structure,
it often makes sense to take $\IA(\widehat \Xi) = \COP(\widehat \Xi)$ in practice.

Now consider the following problem by replacing $\COP(\widehat \Xi \times \RR_+^n)$ with
$\IA(\widehat \Xi \times \RR_+^n)$ in (\ref{equ:vdp_cop}). 
{\small
\begin{equation} \label{equ:vdp_sdp}
\begin{array}{lll}
\bar v_{\D(\hP, \epsilon)}^+= & \min & \lambda \epsilon^2 + \frac1N \sum \limits_{i=1}^{N} (\alpha^i + r\rho^i) \\ 
                & \st &  \alpha^i g_1 g_1^T - H^i(\lambda) + E^T \Diag(u^i)E + \sum \limits_{j \in \B} v^i_{j}Q_j + \rho^i I \in \IA(\widehat \Xi \times \RR_+^n) \ \forall \, i \in [N] \\ 
                & & \rho^i \ge 0  \ \ \ \forall \, i \in [N] \\
                & & \lambda \ge 0.
\end{array}
\end{equation}
}
Obviously, we have the following result:
\begin{theorem}
$v_{\D(\hP_N, \epsilon)}^+ \le \bar v_{\D(\hP, \epsilon)}^+$. 
\end{theorem}

\section{Numerical Experiments} \label{sec:example}

In this section, we validate our proposed Wasserstein-ball
approach (WB) on three applications. We will compare WB
with the moment-based approach (MB) proposed in 
\cite{Natarajan.Teo.Zheng.2011} where the exact values of the 
first two moments of the distributions are known. 
In practice, the moments of the distribution are often not known 
exactly. To this end, Delage and Ye \cite{Delage.Ye.2010} proposed
a data-driven approach to handle this case. However, in this paper, 
we assume that the moments are known exactly for MB. Actually, this 
choice favors MB, but the goal of our experiments is to demonstrate that 
our approach provides a valid upper bound that gets closer to $v_{\PP}$ as
the size of the data set increases, while the MB provides an upper bound,
which does not improve with the size of the data set.

All computations are conducted with Mosek version 8.0.0.28 beta \cite{mosek} on 
an Intel Core i3 2.93 GHz Windows computer with 4GB of RAM and are implemented 
using the modeling language YALMIP \cite{lofberg2004yalmip} in MATLAB  (R2014a) 
version 8.3.0.532. In order to demonstrate the effectiveness of WB, we also implement
a Monte Carlo simulation-based approach (SB) which requires a sufficiently large
number of randomly generated samples. For the project management problem in
Section \ref{sec:pmp}, a linear program is solved for each sample of the Monte Carlo 
simulation, while for the knapsack problem in Section \ref{sec:kp}, an integer
program is solved for each sample. We employ CPLEX 12.4 to solve these 
linear programs and integer programs.

\subsection{Statistical sensitivity analysis of highest-order statistic}  \label{sec:ssa}

The problem of finding the maximum value from a set
$\zeta = (\zeta_1, \ldots, \zeta_n)$ of $n$ numbers can be formulated 
as the optimization problem:
\begin{equation} \label{equ:ssa}
\max \left\{ \zeta^Tx : e^T x = 1, \ x \ge 0\right\}.
\end{equation}
For example, suppose $\zeta_1 = \max \{\zeta_1, \ldots, \zeta_n \}$, 
then the optimal solution to (\ref{equ:ssa}) is $x_1^* = 1, x_2^* =
\cdots = x_n^* = 0$. For the statistical sensitivity analysis problem,
we consider a random vector $\zeta$ following a joint distribution $\PP$. 
In the situation where the true distribution is not known exactly,
our focus is to investigate the upper bound on the expected
maximum value over an ambiguity set containing distributions
that possess partial shared information. 

We consider an instance with $n = 3$ and the true distribution
$\PP$ of $\zeta$ is assumed to be jointly lognormal with first and 
second moments given by $\mu_{\text{log}} \in \RR^3$ and 
$\Sigma_{\text{log}} \in \SYM^3$, respectively. 

In our experiments, we use the following procedure to randomly
generate $\mu_{\text{log}}$ and $\Sigma_{\text{log}}$. We first sample 
$\mu \in \RR^3$ from a uniform distribution $[0, 2]^3$. Then, we 
randomly generate a matrix $\Sigma \in \SYM^3$ as follows:
we set the vector of standard deviations to $\sigma = \tfrac14 e \in \RR^3$,
sample a random correlation matrix $C \in \SYM^3$ using the MATLAB
command `gallery(`randcorr',3)', and set $\Sigma = \diag(\sigma)C\diag(\sigma)+\mu\mu^T$. 
We set $\mu$ and $\Sigma$ as the first and second moments respectively
of the corresponding normal distribution of $\PP$. Then $\mu_{\text{log}}$ 
and $\Sigma_{\text{log}}$ can be computed based on the following 
formulae \cite{Halliwell.2015}:
\begin{equation} \label{equ:moments}
\begin{array}{l}
(\mu_{\text{log}})_i = e^{\mu_i + 0.5\Sigma_{ii}}, \\
(\Sigma_{\text{log}})_{ij} = e^{\mu_i + \mu_j + 0.5(\Sigma_{ii} + \Sigma_{jj})}(e^{\Sigma_{ij}}-1).
\end{array}
\end{equation}

We can cast this problem into our framework by setting
$m = 1, k = n+1$, $\xi = (1, \zeta_1, \ldots, \zeta_n)$,
$F = (0, I)$, and $\B = \emptyset$. Obviously, Assumptions
\ref{ass:feasible_bound} and \ref{ass:convex} are satisfied. 
Assumption \ref{ass:binary} is vacuous. Although Assumption 
\ref{ass:uncertainty_bound} does
not hold, problem (\ref{equ:vdp_sdp}) can still provide
a valid upper bound on the expected optimal value as 
discussed in Section \ref{sec:ubd}.

\subsubsection{The deviation of empirical Wasserstein radii}

In this experiment, we consider a particular underlying
distribution $\PP$ that is generated by the procedure mentioned
above. Also, we consider eight cases for the size
of the dataset: $N \in \{10$, $20$, $40$, $80$, $160$, $320$, $640$,
$1280 \}$. For each case, we randomly generate a dataset ${\widehat \Theta}_N$ 
containing $N$ independent samples
from $\PP$ and use the procedure in Section \ref{sec:wass}
to determine a desired radius from a pre-specified set 
$\mathcal E = \{0.001$, $0.005$, $0.01$, $0.05$, $0.1$, $0.2$, 
$0.3$, $0.4$, $0.5$, $0.6$, $0.7$, $0.8$, $0.9$, $1.0$, $2.0\}$\footnote{
From preliminary experiments, the largest element $2.0$ in 
set $\cal E$ returned 1 as the empirical confidence level for all the 
experiments we conducted. Thus, we believe it is sufficient to have
$2.0$ as the largest element here.}. 
In particular, we set $K = 100$ in Algorithm \ref{algorithm:a1}.  
Figure \ref{Fig:epsilon} shows the trend of the reliabilities over different
Wasserstein radii for $N \in \{ 20$, $80$, $320$, $1280 \}$. Clearly, 
smaller Wasserstein radii tend to have lower empirical confidence 
levels. Furthermore, as the sample size increases, the empirical confidence
level increases as well for the same Wasserstein radius. The result of this
experiment indicates that we can practically choose a Wasserstein
radius with a desired statistical guarantee for each case of $N$. 
We remark that the derived radii can be used for datasets of
the same sizes generated from different distributions of the same family
(lognormal distributions in this application).


\begin{figure}[ht]
  \centering
  \includegraphics[width=5.0in]{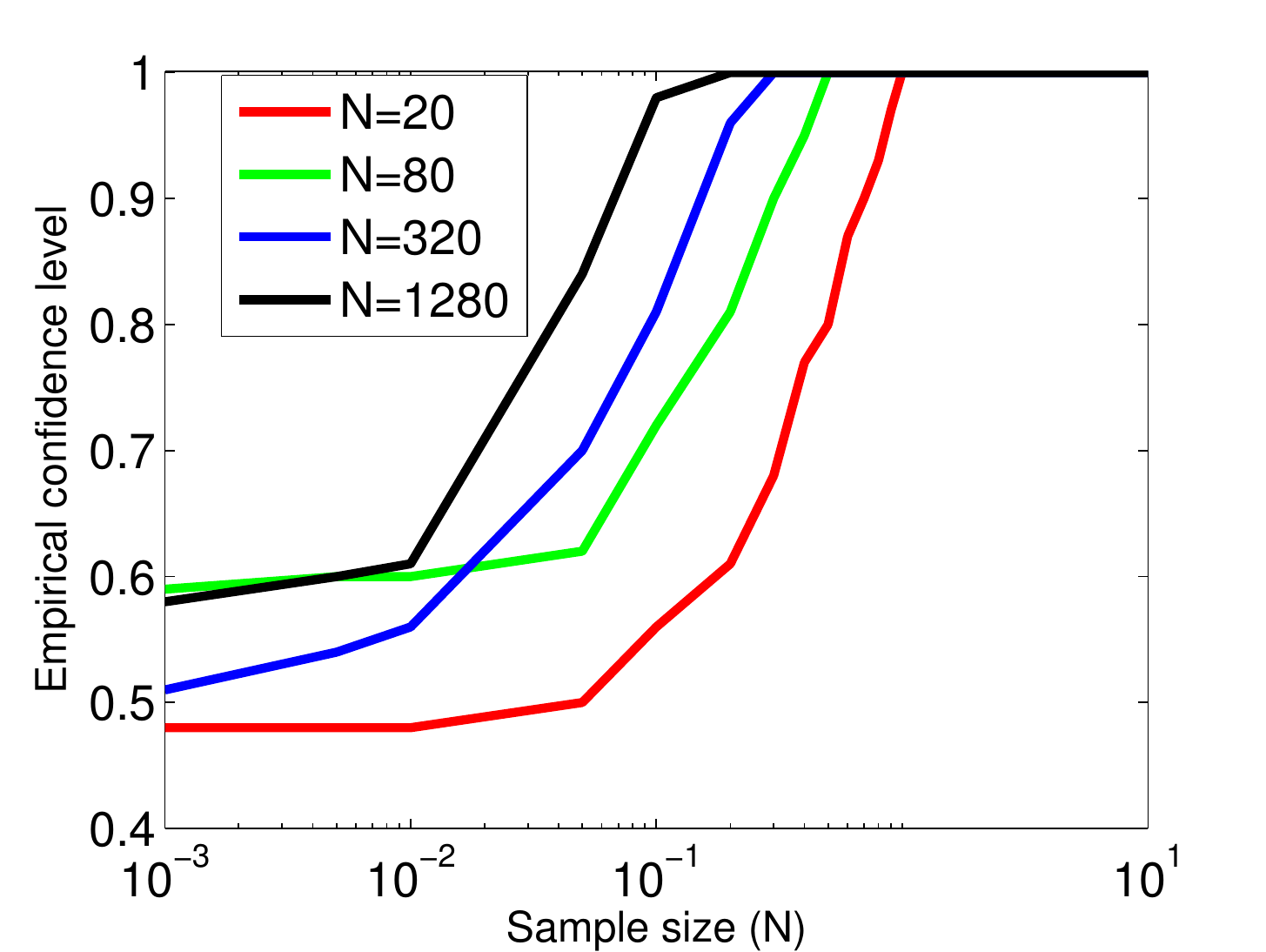}
  \caption{Empirical confidence levels of different Wasserstein radii for $N \in \{20, 80, 320$, $1280 \}$ respectively.
}
  \label{Fig:epsilon}
\end{figure}

\subsubsection{Instances with the same underlying distribution} \label{sec:ssa_single}

Our next experiment is to focus on a particular joint lognormal
distribution $\PP$. We consider eight cases: $N \in \{ 10$, $20$, $40$, $80$, $160$, $320$,
$640$, $1280 \}$. For each case, we test $100$ trials and in each
trial we randomly generate $N$ independent samples from $\PP$ 
and choose the Wasserstein radius with an empirical confidence level of $0.90$.
We compare our approach with MB where the first
two moments are directly given as $\mu_{\text{log}}$ and $\Sigma_{\text{log}}$. 
We also randomly generate $100000$ independent
samples from $\PP$ to simulate the true expected optimal value.

We demonstrate experimental results in Figure  \ref{Fig:single_dist}. 
Note that the solid black line represents the simulated value of
the true expected optimal value, while the dashed black line 
represents the upper bound calculated by the moment-based 
approach. Furthermore, we solve an instance of (\ref{equ:vdp_sdp}) 
for each of the $100$ trails in each case of $N$. We use the blue, red,
and green lines to respectively represent the $80^{th}$ quantile,
the median, and the $20^{th}$ quantile of the values from the 
$100$ trials in each case. Figure \ref{Fig:single_dist} shows that our approach provides
weaker bounds on the expected optimal value for smaller sample
sizes. However, as the size of samples increases, our approach
provides stronger bounds and the bounds get relatively close to the
simulated value. In contrast, the value from MB remains the same 
regardless of sample size. 

\begin{figure}[ht]
  \centering
  \includegraphics[width=5.0in]{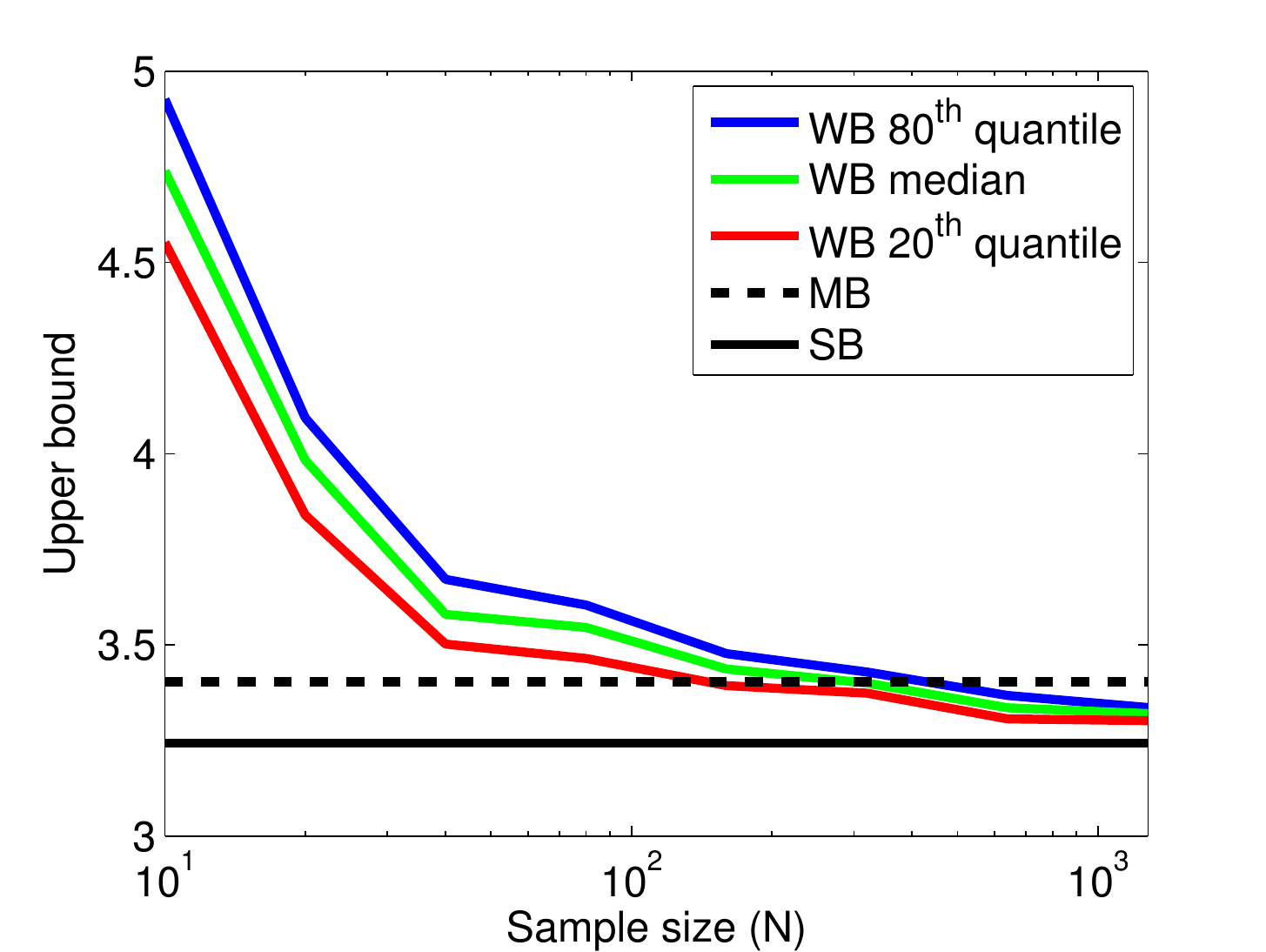}
  \caption{The comparison of WB and MB for the stochastic sensitivity analysis problem over different
sample sizes for a particular randomly generated underlying distribution. 
}
  \label{Fig:single_dist}
\end{figure}

\subsubsection{Instances with different underlying distributions}

In this experiment, we consider eight cases $N \in \{ 10$, $20$, $40$, 
$80$, $160$, $320$, $640$, $1280 \}$. In each case, we randomly 
generate $100$ trials. For each trial in each case we generate $N$
samples from a random lognormal distribution whose first and second 
moments are generated by using the procedure at the beginning of
this section. For each trial in each case, we solve an instance of 
(\ref{equ:vdp_sdp}) with a Wasserstein radius corresponding to 
an empirical confidence level of $0.90$. We also simulate the true expected 
optimal values by randomly generating $100000$ samples from the
true distributions.
 
For each trial in each case, we denote the optimal value from (\ref{equ:vdp_sdp}) 
by $\bar v_{\text{WB}}^+$ and the simulated value by $v_{SB}$. 
Then, we calculate the relative gap between WB and SB as 
\[
\text{gap(WB)} := \dfrac{\bar v_{\text{WB}}^+ - v_{\text{SB}}}{v_{\text{SB}}}.
\] 
We take the average of the relative gaps over the $100$ trials for each case. 
Then, for each trial in each case, we solve MB with the first two
moments computed by (\ref{equ:moments}). Denote the optimal value
from MB by $\bar v_{\text{MB}}^+$. Similarly, we calculate the relative gap 
between MB  and SB as
\[
\text{gap(MB)} := \dfrac{\bar v_{\text{MB}}^+ - v_{\text{SB}}}{v_{\text{SB}}}.
\] 
We then take the average of the relative gaps over the $100$ trials in each case. 
Figure \ref{Fig:multi_dist} illustrates the average relative gaps from both WB and
MB over the eight cases. Clearly, the upper bound from WB approaches the simulated
value along with the increase of the size of samples, while the average relative gap 
between the bound from MB and the simulated value does not. 
\begin{figure}[ht]
  \centering
  \includegraphics[width=5.0in]{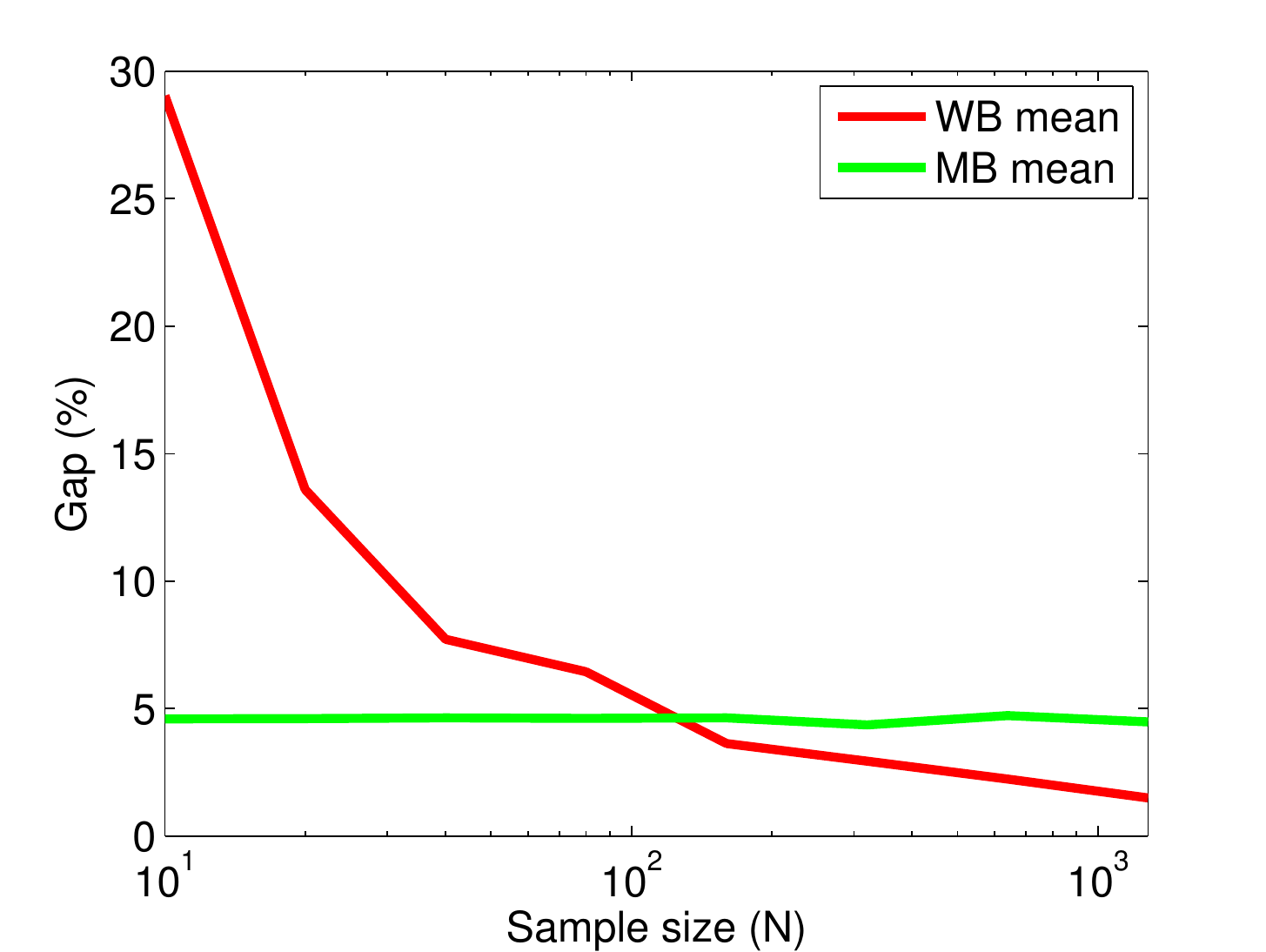}
  \caption{The average gaps from MB and WB for the stochastic sensitivity analysis problem 
over the eight cases: $N \in \{ 10$, $20$, $40$, $80$, $160$, $320$, $640$, $1280 \}$.
The blue line represents the average relative gap between the optimal value from WB and 
the simulated value; the red line represents the average relative gap between the optimal value from MB and 
the simulated value.}
  \label{Fig:multi_dist}
\end{figure} 

Table \ref{Tab:real_reliability} shows 
the percentage of the $100$ trials where the optimal values
from WB are greater than or equal to the corresponding simulated optimal values in the eight cases.
The result demonstrates that the derived empirical Wasserstein radii indeed
provide desired statistical guarantees in practice. 
\begin{table}[tbp]
\centering
\begin{tabular}{c|cccccccc}  
Case number & 1 & 2 & 3 & 4 & 5 & 6 & 7 & 8   \\ \hline
Empirical confidence level & 1.00 & 1.00 & 1.00 & 1.00 & 0.97 & 0.98 & 0.97 & 0.98  
\end{tabular}
\caption{The percentage of the $100$ trials where the optimal 
values from WB are greater or equal to the simulated values over the 
$8$ cases for the stochastic sensitivity analysis problem.}
\label{Tab:real_reliability}
\end{table}


\subsection{Project management problem} \label{sec:pmp}

In this application, we consider a project management problem,
which can be formulated as a longest-path problem on a directed
acyclic graph. The arcs denote activities and the nodes denote the
completions of a set of activities. Arc lengths denote the time
to complete the activities. Thus, the longest path from the starting
node $s$ to the ending node $t$ gives the time needed to compete
the whole project. Let $\zeta_{ij}$ be the length (time) of arc (activity)
from node $i$ to node $j$. The problem can be solved as a linear
program due to the network flow structure as follows:

\begin{equation}
\begin{array}{ll} 
\max & \sum \limits_{(i,j) \in {\cal A}} \zeta_{ij} x_{ij} \\
\st     & \sum \limits_{i : (i,j) \in {\cal A}} x_{ij} - \sum \limits_{j : (i,j) \in {\cal A}} x_{ji} = 
\left\{ 
\begin{array}{l}
 1, \ \ \ \ \text{if $i = s$} \\
 0, \ \ \ \ \text{if $i \in {\cal N}$, and $i \neq s,t$} \\
-1, \ \ \ \ \text{if $i = t$}  
\end{array}
\right.  \\
& x_{ij} \ge 0, \ \forall \; (i,j) \in {\cal A},
\end{array}
\end{equation}
where ${\cal A}$ denotes the set containing all the arcs, ${\cal N}$
denotes the set containing all nodes on the network, and $x_{ij}$ denotes
the number of units of flow sent from node $i$ to node $j$ through arc 
$(i,j) \in {\cal A}$. For the stochastic project
management problem, the activity times are random. In such cases,
due to the resource allocation and management constraints, the project manager
would like to quantify the worst-case expected completion time of the project,
which is corresponding to the worst-case longest path of the network. 

We consider an instance with a network structure shown in Figure \ref{Fig:network}.
\begin{figure}[ht]
  \centering
  \includegraphics[width=3.0in]{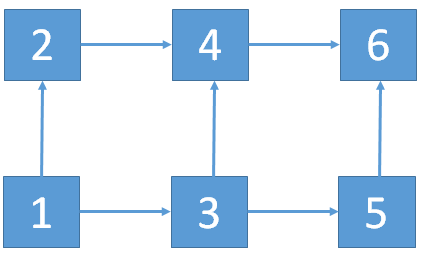}
  \caption{The structure of the project network where $1$ and $6$ 
                are the starting and ending nodes respectively.}
  \label{Fig:network}
\end{figure} 
This network consists of $7$ arcs and $6$ nodes. There
are $3$ paths from the starting node to the ending node on the network.
In the experiments of this example, we consider truncated joint normal
distributions. We use the following procedure to generate a truncated joint
normal distribution $\PP$: denoting $|\cal A|$ by the cardinality of set $\cal A$, 
we generate $\zeta \ge 0$ from a jointly normal
distribution with first and second moments given by $\mu \in \RR^{|{\cal A}|}$
and $\Sigma \in \SYM^{|{\cal A}|}$, respectively. Specifically, we sample
$\mu$ from a uniform distribution $[0, 5]^{|{\cal A}|}$ while the matrix
$\Sigma$ is generated randomly using the following procedure: we set the
vector of standard deviations to $\sigma = e$, sample a random correlation
matrix $C \in \SYM^{|{\cal A}|}$ using the MATLAB command
`gallery(`randcorr',$|{\cal A}|$)', and set $\Sigma = \diag(\sigma)C\diag(\sigma)+\mu\mu^T$. 
Skipping the details, we can cast the network flow problem
into our framework. It is straightforward to check that Assumptions 
\ref{ass:feasible_bound}, \ref{ass:convex}, and \ref{ass:uncertainty_bound}
are satisfied and Assumption \ref{ass:binary} is vacuous. 

\subsubsection{Instances with the same underlying distribution}

The first experiment of this example focuses on a particular underlying distribution
$\PP$. We consider seven cases: $N \in \{10$, $20$, $40$, $80$, $160$, $320$,
$640\}$. For each case, we run $100$ trials and in each trial we randomly
generate a dataset $\widehat \Theta_N$ containing $N$ independent samples from $\PP$. 
We use the procedure in Section \ref{sec:wass} to
compute an empirical confidence level set for each case. Then, we use computed
empirical confidence level sets to derive empirical Wasserstein radii for the following computations.  
For each trial in each case, we solve an instance of (\ref{equ:vdp_sdp})
with a Wasserstein radius corresponding to an empirical confidence level of $0.90$. 
We compare WB with MB where the first two moments are approximated
by using the sample mean and variance from $100000$ samples.
The computed moments are close to their theoretical counterparts as the 
sample size is considerably large. We also simulate the expected optimal 
value over the $100000$ samples. Figure \ref{Fig:single_dist_pnp}
shows that WB provides weaker bounds on the expected optimal value for 
smaller sample sizes. However, as the size of samples increases, WB provides
stronger bounds and the bounds get relatively close to the
simulated value. In contrast, the bounds from MB remains the same 
regardless of the change of sample sizes. 

\begin{figure}[ht]
  \centering
  \includegraphics[width=5.0in]{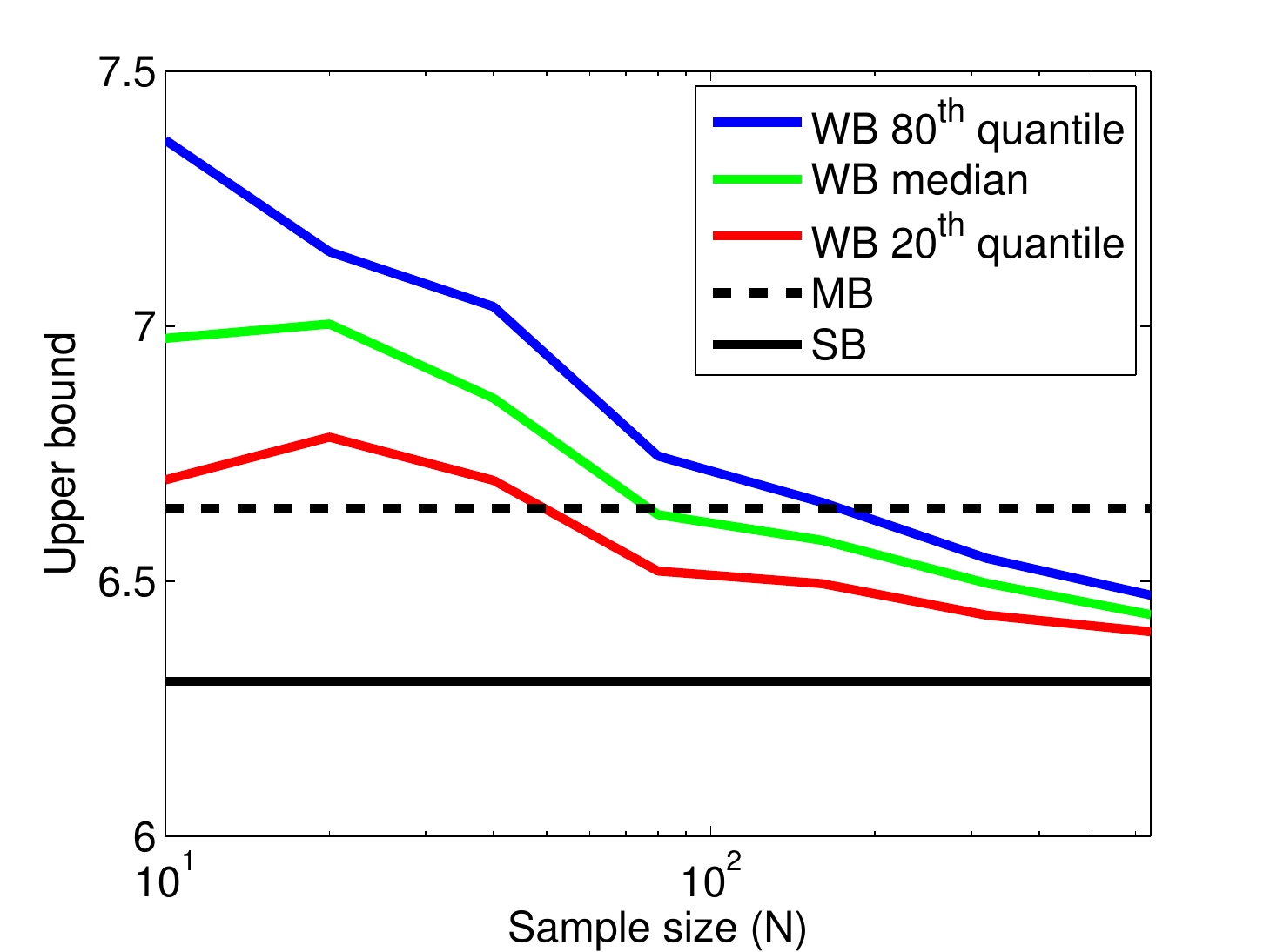}
  \caption{The comparison of WB and MB  for the project management problem 
over different sample sizes for a particular randomly generated underlying distribution. 
}
  \label{Fig:single_dist_pnp}
\end{figure}

\begin{figure}[ht]
  \centering
  \includegraphics[width=5.0in]{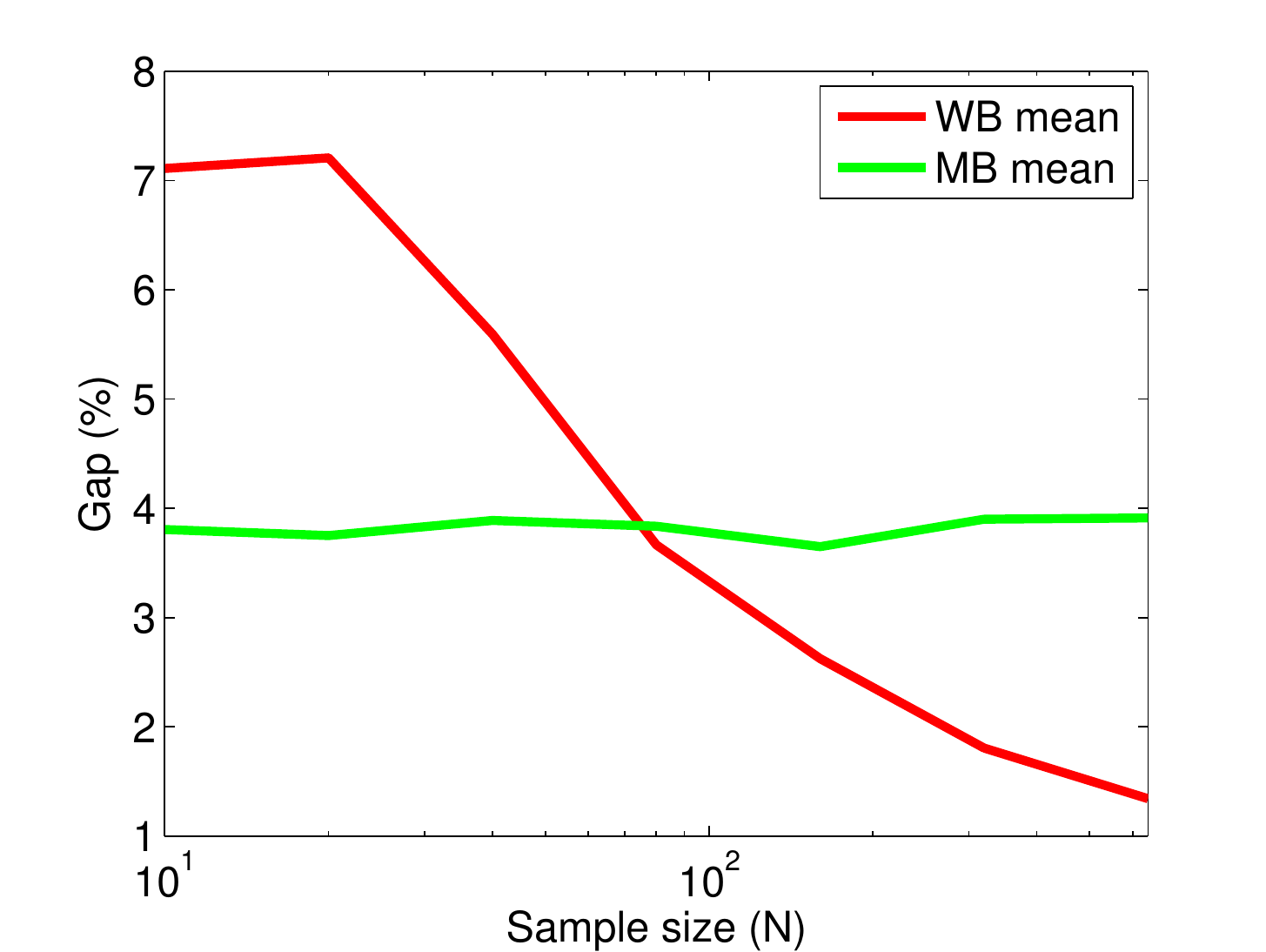}
  \caption{The average gaps from both MB and WB for the project management problem over the seven cases:
$N \in \{10$, $20$, $40$, $80$, $160$, $320$, $640\}$.
}
  \label{Fig:multi_dist_pnp}
\end{figure} 

\subsubsection{Instances with different underlying distributions}

In this experiment, we consider seven cases: $N \in \{ 10$, $20$, $40$, $80$, 
$160$, $320$, $640\}$. For each case, we randomly generate $100$ trials in 
which $N$ independent samples are drawn from a randomly generated 
truncated joint normal distribution. Then, for each trial in each case, we solve 
and instance of (\ref{equ:vdp_sdp}) with a Wasserstein 
radius corresponding to a $0.90$ empirical confidence level. 
We solve MB where the first two moments are approximated 
by computing the sample mean and variance of $100000$ samples. 
We also simulate the expected optimal value over the $100000$ samples 
for each trial in each case. We compute the relative gap between the WB and
SB as well as the relative gap between MB and SB. 
Then, for each case, we take the average of the relative
gaps from both WB and MB over the $100$ trials. 
Figure \ref{Fig:multi_dist_pnp} illustrates the average relative gaps 
over the seven cases. 
Clearly, the upper bound from WB approaches to the simulated
value along with the increase in the size of samples, while the gap between
the bound from MB and the simulated value remains relatively the same
as the sample size increases. Table \ref{Tab:real_reliability_pnp} shows
the percentage of the $100$ trials where the optimal values
from WB are greater than or equal to the corresponding
simulated optimal values over the seven cases.  

\begin{table}[tbp]
\centering
\begin{tabular}{c|ccccccc}  
Case number & 1 & 2 & 3 & 4 & 5 & 6 & 7    \\ \hline
Empirical confidence level & 1.00 & 1.00 & 1.00 & 1.00 & 0.97 & 0.98 & 0.97   
\end{tabular}
\caption{the percentage of the $100$ trials where the optimal values
from WB are greater than or equal to the corresponding
simulated optimal values over the $7$ cases for
the project management problem.}
\label{Tab:real_reliability_pnp}
\end{table}

\subsection{Knapsack problem} \label{sec:kp}

A standard knapsack problem is defined as follows: given a set of items
and each with a weight and a value, the problem is to determine the number 
of items to include in a knapsack such that the total weight is less than or equal to
a given capacity limit and the total value is maximized; see the detail in 
\cite{Martello.Toth.1990}. Let $w_i$ and $\zeta_i$ be the weight and value 
of item $i \ (i = 1, \ldots, n)$ respectively. Let $W$ be the maximum weight 
capacity of the knapsack. Then, the knapsack problem can be formulated 
as an integer program:
\[
v(\zeta) := \max \left \{ \sum_{i = 1}^n \zeta_ix_i : \sum_{i=1}^n w_i x_i \leq W, \ x_i \in \{0, 1\} \right\},
\]
where $x_i$ represents the number of item $i$ to include in the knapsack. 
Assume that the values of the items are random and follow an unknown 
joint distribution. Assume also that we can collect a dataset containing $N$ 
samples with each corresponding to an observation of the $n$ item values. 
In such cases, we would like to compute a data-driven distributionally robust 
upper bound on the expected maximum value of the knapsack. We can 
approximate the upper bound by solving problem (\ref{equ:vdp_sdp}).

We consider an instance with $n=4, w = (5, \ 4, \ 6, \ 3 )^T$,
and $W = 10$. The true distribution $\PP$ of $\zeta$ is assumed to be jointly lognormal
with first and second moments given by $\mu_{\text{log}} \in \RR^4$ and 
$\Sigma_{\text{log}} \in \SYM^4$, respectively. Similar to the procedure described in 
Section \ref{sec:ssa}, we sample $\mu \in \RR^4$ from a uniform distribution $[0, 2]^4$. 
Then, we randomly generate a matrix $\Sigma \in \SYM^4$ as follows:
we set the vector of standard deviations to $\sigma = \tfrac14 e \in \RR^4$,
sample a random correlation matrix $C \in \SYM^4$ using the MATLAB
command `gallery(`randcorr',4)', and set $\Sigma = \diag(\sigma)C\diag(\sigma)+\mu\mu^T$. 
Then $\mu_{\text{log}}$ and $\Sigma_{\text{log}}$ can be computed
based on (\ref{equ:moments}). We can easily cast this problem into
our framework. For simplicity, we skip the details. It is also
straightforward to check that the conditions in Assumptions \ref{ass:feasible_bound},
\ref{ass:binary}, and \ref{ass:convex} are satisfied. Although 
Assumption \ref{ass:uncertainty_bound} is not satisfied, we still can solve
(\ref{equ:vdp_sdp}) to obtain a valid upper bound on the expected optimal
value of the knapsack problem. 

\subsubsection{Instances with the same underlying distribution}

In the first experiment, we focus on a particular underlying distribution
$\PP$ and consider eight cases: $N \in \{10$, $20$, $40$, 
$80$, $160$, $320$, $640$, $1280 \}$. For each case, we run $100$ 
trials and in each trial we randomly generate a dataset $\widehat \Theta_N$ 
containing $N$ independent samples from $\PP$. Similarly, we derive
empirical Wasserstein radii for each case. Then, for each trial in each
case, we solve an instance of (\ref{equ:vdp_sdp}) with an empirical Wasserstein radius 
corresponding to an empirical confidence level of $0.90$. We compare our approach 
with MB where the first two moments are computed by (\ref{equ:moments}). 
We simulate the expected optimal value over $100000$ samples. 
Note that we solve an integer program for each sample in the simulation. 
Figure \ref{Fig:kp_single} shows that WB provides weaker bounds on the 
expected optimal value for smaller sample sizes. However, as the size of 
samples increases, WB provides stronger bounds and the bounds get 
relatively close to the simulated value. In contrast, the bounds from MB 
remains the same regardless of the change of sample sizes. We remark
that the upper bound computed by WB may not converge to the true
expected optimal value as the sample size increases to infinity; see the trend 
shown in Figure \ref{Fig:kp_single}. This is due to the fact that problem 
(\ref{equ:vdp_sdp}) is a relaxation of problem (\ref{equ:vdp_cop}) and the
fact the relaxation is not tight in this example. 

\begin{figure}[ht]
  \centering
  \includegraphics[width=5.0in]{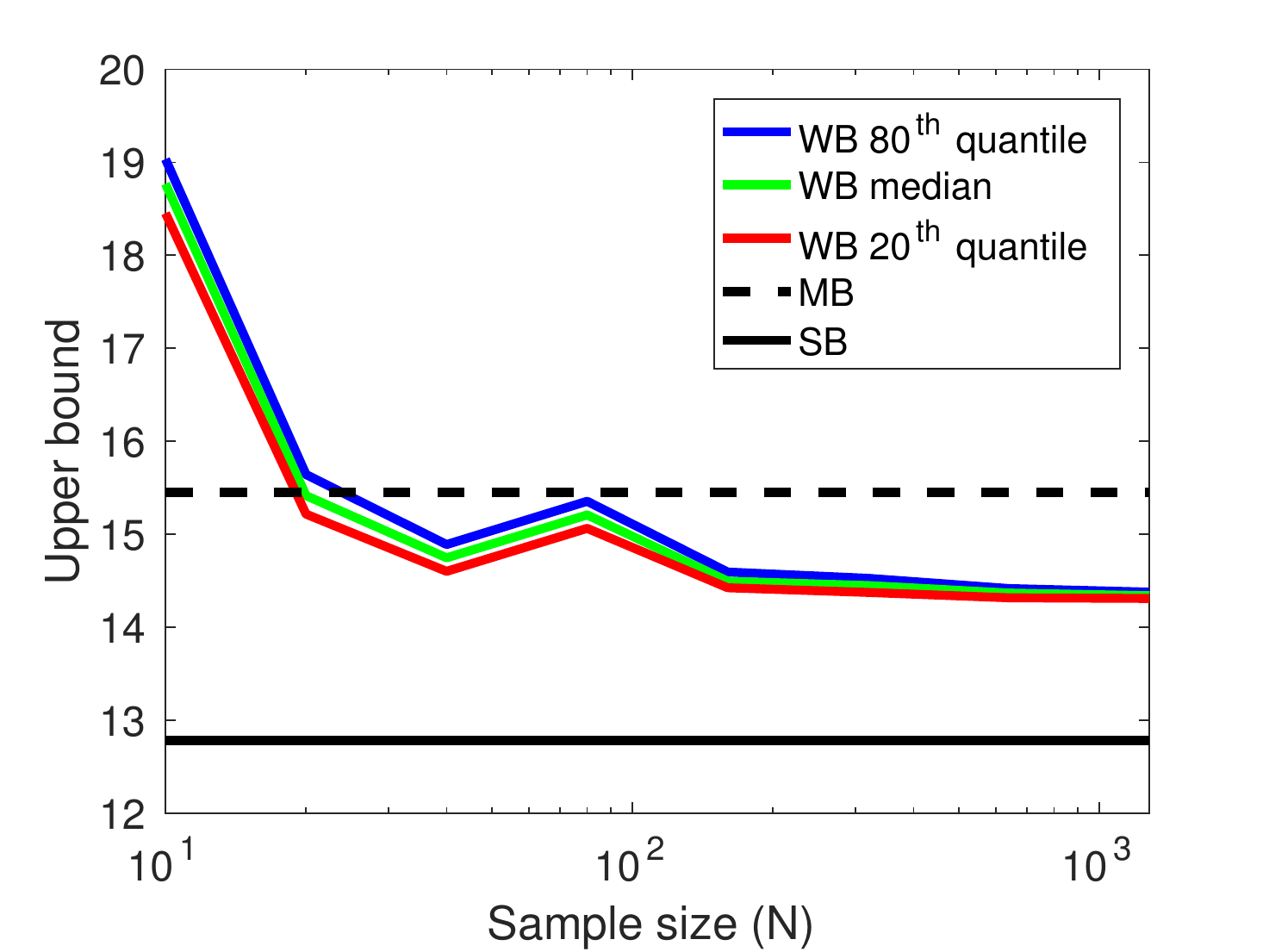}
  \caption{The comparison of WB and MB for the knapsack problem over different
sample sizes for a particular randomly generated underlying distribution. }
  \label{Fig:kp_single}
\end{figure}

\begin{figure}[ht]
  \centering
  \includegraphics[width=5.0in]{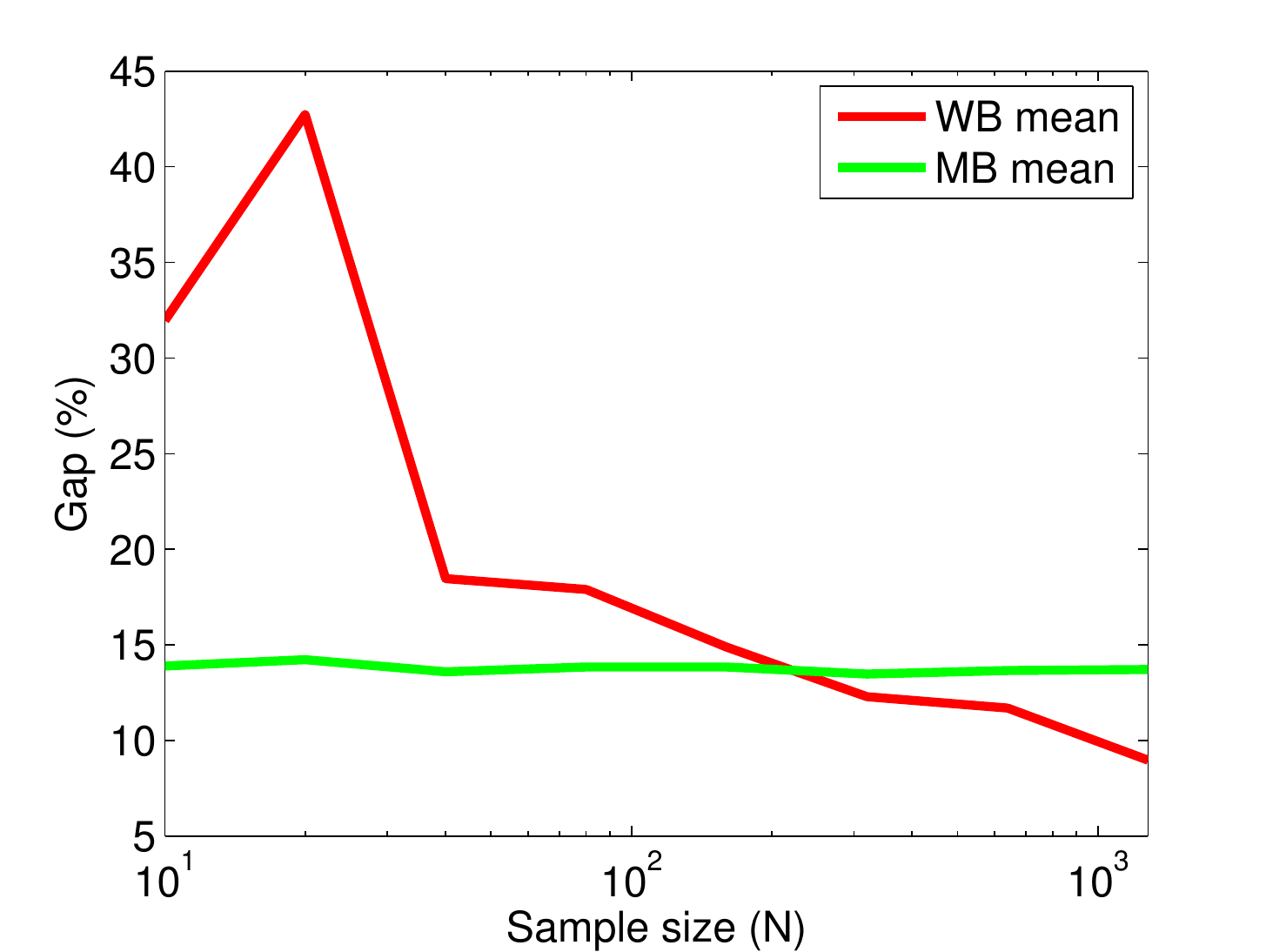}
  \caption{Illustration of the average gaps from both MB and WB in the case of
$N \in \{10$, $20$, $40$, $80$, $160$, $320$, $640\}$.
The blue line represents the average gap between the optimal values from WB and 
the simulated values;
the red line represents the average gap between the optimal values from MB and 
the simulated values.}
  \label{Fig:kp_multi}
\end{figure} 

\subsubsection{Instances with different underlying distributions}

This experiment considers eight cases: $N \in \{ 10$, $20$, $40$, $80$, $160$, 
$320$, $640$, $1280 \}$. For each case, we randomly generate 
$100$ trials with each trial is drawn from a randomly generated 
joint lognormal distribution. Then, for each trial in
each case, we solve an instance of (\ref{equ:vdp_sdp}) with a Wasserstein 
radius corresponding to a $0.90$ empirical confidence level. 
We simulate the expected optimal value over $100000$ samples 
for each trial in each case. Next, we compute the relative gap 
between the values of WB and SB as well as the relative gap between
the values of MB and SB. Then, for each case, we take the average of 
the relative gaps from both WB and MB over the $100$ trials. 
Figure \ref{Fig:kp_multi} illustrates the average relative gaps over 
the seven cases. Clearly, the WB gap becomes narrower as the 
sample size increases, while the MB gap remains relatively the same. 

Table \ref{Tab:real_reliability_kp} shows
the percentage of the $100$ trials where the optimal values
from WB are greater than or equal to the corresponding
simulated optimal values over the eight cases.  

\begin{table}[tbp]
\centering
\begin{tabular}{c|cccccccc}  
Case number & 1 & 2 & 3 & 4 & 5 & 6 & 7 & 8   \\ \hline
Empirical confidence level & 1.00 & 1.00 & 1.00 & 1.00 & 1.00 & 1.00 & 1.00 & 1.00   
\end{tabular}
\caption{The percentage of the $100$ trials where the optimal values
from WB are greater than or equal to the corresponding simulated optimal values over the $8$ cases for
the knapsack problem.}
\label{Tab:real_reliability_kp}
\end{table}

\section{Concluding Remarks} \label{sec:conclusion}

In this paper, we have studied the expected optimal value of a mixed 0-1
programming problem with uncertain objective coefficients following a
joint distribution whose information is not known exactly but a set 
of independent samples can be collected. Using the samples, we
have constructed a Wasserstein-based ambiguity set that contains
the true distribution with a desired confidence level. We proposed
an approach to compute the upper bound on the expected optimal
value. Then under mild assumption, the problem was reformulated
to a copositive program, which leads to a semidefinite-based relaxation. 
We have validated the effectiveness of our approach over three
applications.

\section*{Acknowledgements}
We would like to thank Kurt Anstreicher, Qihang Lin, Luis Zuluaga,
and Tianbao Yang for many useful suggestions, which helped us
improve the results of the paper.



\end{onehalfspace}

\bibliographystyle{plain}

\bibliography{distr}



\end{document}